\documentclass[psamsfonts,12pt]{amsart}
\usepackage[margin=1in]{geometry}  
\usepackage{graphicx}              
\usepackage{amsmath}               
\usepackage{amsfonts}              
\usepackage{amsthm}                
\usepackage{verbatim}              
\usepackage{indentfirst}           
\usepackage{underscore}
\usepackage{enumitem}
\usepackage{mathtools} 
\usepackage{amssymb}
\usepackage[all]{xy}
\makeatletter
\renewcommand*\env@matrix[1][*\c@MaxMatrixCols c]{%
  \hskip -\arraycolsep
  \let\@ifnextchar\new@ifnextchar
  \array{#1}}
\makeatother

\newtheorem{thm}{Theorem}[section]
\newtheorem{lem}[thm]{Lemma}
\newtheorem{prop}[thm]{Proposition}

\theoremstyle{remark}
       \newtheorem{rmk}{Remark}
\theoremstyle{remark}

\newcommand*{\Scale}[2][4]{\scalebox{#1}{$#2$}}%

 \newcommand{\Orb}{\operatorname{Orb}}
  \newcommand{\Rat}{\operatorname{Rat}}

 \newcommand{\PostCrit}{\operatorname{PostCrit}} 
 \newcommand{\ord}{\operatorname{ord}}  
    
    \newcommand{\Res}{\operatorname{Res}}  
    \newcommand{\sep}{\operatorname{sep}} 
\usepackage{amssymb,fge}
\newcommand{\mysetminus}{\mathbin{\fgebackslash}}

\title{Integrality estimates in orbits over function fields}
\author[Wade Hindes]{Wade Hindes}
\address{Department of Mathematics, The Graduate Center, City University of New York (CUNY); 365 Fifth Avenue, New York, NY 10016, USA}
\email{whindes@gc.cuny.edu}
\begin{document}
\maketitle
\begin{abstract} \normalsize We prove a version of Silverman's dynamical integral point theorem for a large class of rational functions defined over global function fields. 
\end{abstract}
\renewcommand{\thefootnote}{}


\section{Introduction} 
Let $K=\mathbb{F}_q(t)$, let $\phi(x)\in K(x)$ be a rational function of degree $d\geq2$, and let
\[\phi(x)=\frac{a_d(t)x^d+a_{d-1}(t)x^{d-1}+\dots a_0(t)}{b_d(t)x^d+b_{d-1}(t)x^{d-1}+\dots b_0(t)}.\] 
To every $\phi$ we associate an explicit $7\times 6$ matrix with coefficients in $K$:\\
\begin{equation*}
\hspace*{-1.05cm} 
M_\phi:=\Scale[0.5]{\begin{pmatrix}
    -a_d^2 &\dots & 0&0\\
     -2a_da_{d-1}& \dots & a_db_{d-1} - a_{d-1}b_d & 0 \\ 
   -2a_da_{d-2} - a_{d-1}^2 & \dots &2a_db_{d-2} - 2a_{d-2}b_d & a_db_{d-1}-a_{d-1}b_d \\
    -2a_da_{d-3} - 2a_{d-1}a_{d-2}& \dots  & 3a_db_{d-3} + a_{d-1}b_{d-2} - a_{d-2}b_{d-1} - 3a_{d-3}b_d& 2a_db_{d-2} - 2a_{d-2}b_d \\
-2a_da_{d-4} - 2a_{d-1}a_{d-3} - a_{d-2}^2 & \dots  & 4a_db_{d-4} + 2a_{d-1}b_{d-3} - 2a_{d-3}b_{d-1} - 4a_{d-4}b_d & 3a_db_{d-3} + a_{d-1}b_{d-2} - a_{d-2}b_{d-1} - 3a_{d-3}b_d \\
-2a_da_{d-5} - 2a_{d-1}a_{d-4} - 2a_{d-2}a_{d-3} & \dots & 5a_db_{d-5} + 3a_{d-1}b_{d-4} + a_{d-4}b_{d-3} - a_{d-3}b_{d-2} - 3a_{d-4}b_{d-1} - 5a_{d-5}b_d & 4a_db_{d-4} + 2a_{d-1}b_{d-3} - 2a_{d-3}b_{d-1} - 4a_{d-4}b_d \\
-2a_da_{d-6} - 2a_{d-1}a_{d-5} - 2a_{d-2}a_{d-4} - a_{d-3}^2& \dots & 6a_db_{d-6} + 4a_{d-1}b_{d-5} + 2a_{d-2}b_{d-4} - 2a_{d-4}b_{d-2} - 4a_{d-5}b_{d-1} - 6a_{d-6}b_d& \;\;5a_db_{d-5} + 3a_{d-1}b_{d-4} + a_{d-2}b_{d-3} - a_{d-3}b_{d-2} - 3a_{d-4}b_{d-1} - 5a_{d-5}b_d\;
\end{pmatrix}}
\end{equation*} 
\\ 
See the proof of Lemma \ref{lem:Riccati} for the three remaining columns. Our main result is the following:
\begin{thm}\label{thm:Silv} Let $K=\mathbb{F}_q(t)$ and suppose that $\phi(x)\in K(x)$ of degree $d\geq2$ has the following properties:  
\begin{enumerate}[topsep=6pt, partopsep=6pt, itemsep=3pt]
\item[\textup{(1)}] The overdetermined system of linear equations corresponding to the augmented matrix $(M_\phi|R_\phi)$, with column vector $R_\phi=(r_n)$ defined by  
\[r_n=\sum_{i=0}^na_{d-i}b_{d-n+i}'-\sum_{i=0}^nb_{d-i}a_{d-n+i}'\]
for $0\leq n\leq6$, is inconsistent (i.e. has no solutions).   
\item[\textup{(2)}] $\infty\not\in\PostCrit_\phi$, i.e. $\infty$ is not in the forward orbit of any of the critical points of $\phi$.     
\end{enumerate} 
Let $\alpha\in K$ be any wandering point for $\phi$, and write 
\[\phi^n(\alpha)=\frac{a_n(t)}{b_n(t)}\]
for some polynomials $a_n(t), b_n(t)\in\mathbb{F}_q[t]$ in lowest terms. Then 
\begin{equation*}{\label{silvlim}}
\limsup_{n\rightarrow\infty}\, \frac{\deg(a_n)}{\deg(b_n)}\leq 2.
\end{equation*}
In particular, $\Orb_\phi(\alpha):=\{\phi^n(\alpha)\}_{n\geq1}$ contains finitely many polynomials. Additionally, 
\begin{equation*}{\label{liminf}}
\frac{1}{2}\leq\liminf_{n\rightarrow\infty}\, \frac{\deg(a_n)}{\deg(b_n)}\leq\limsup_{n\rightarrow\infty}\, \frac{\deg(a_n)}{\deg(b_n)}\leq 2 \vspace{.1cm}
\end{equation*} 
whenever the rational function $1/\phi(1/x)$ also satisfies assumptions \textup{(1)} and \textup{(2)}.      
\end{thm} 
We make the convention that any coefficients defining $M_\phi$ and $R_\phi$ with negative indices are zero. Moreover, we say that $\alpha\in K$ is a \emph{wandering point} for $\phi$ if the dynamical orbit $\Orb_\phi(\alpha):=\{\phi^n(\alpha)\}_{n\geq1}$ is an infinite set. For examples and computations related to Latt\'{e}s maps and elliptic curves, see Section \ref{sec:lattes}. 
\section{Diophantine Approximation Over Function Fields}
For number fields $K/\mathbb{Q}$ and rational functions $\phi(x)\in K(x)$, the main result on diophantine approximation of iterates is due to Silverman \cite{Silv-Int}. Loosely speaking, this result states that if $0$ and $\infty$ are not totally ramified fixed points of the second iterate of $\phi$ (in particular, we do not consider polynomials), then the numerator and denominator of the iterates of any wandering point $\alpha\in K$ have roughly the same height. We state this precisely over the rational numbers:   
\begin{thm}[Silverman]{\label{Silv-NF}} Let $\phi(x)\in\mathbb{Q}(x)$ be a rational map with the property that both $\phi^2(x)$ and $1/\phi^2(1/x)$ are not polynomials. Let  $\alpha\in\mathbb{Q}$ be a wandering point for $\phi$,  and write \[\phi^n(\alpha)=\frac{a_n}{b_n}\] for some integers $a_n,b_n\in\mathbb{Z}$ in lowest terms. Then 
\[\lim_{n\rightarrow\infty} \frac{\log|a_n|}{\log|b_n|}=1.\]
In particular, $\Orb_\phi(\alpha)$ contains only finitely many integers. 
\end{thm}
\begin{rmk} Theorem \ref{Silv-NF} generalizes a result of Siegel \cite{Siegel} regarding the height of the numerator and denominator of the $x$-coordinate of large point on an elliptic curve \cite[Theorem 3.39]{Silv-Dyn}.    
\end{rmk} 
The main obstruction to generalizing Silverman's proof in characteristic zero to positive characteristic is the failure of Roth's theorem. For instance, it was known to Mahler \cite{Mahler} that the algebraic function $\beta=\sum_{j=0}^\infty t^{-q^j}$, satisfying the equation $\beta^q-\beta-t^{-1}=0$, has a large diophantine approximation exponent: \vspace{.075cm}   
\[E(\beta):=\limsup\Big(-\frac{\log_q|\beta-P/Q|}{\log_q|Q|}\,\Big)=q=d(\beta):=[K(\beta):K];\] 
here $P, Q\in \mathbb{F}_q[t]$ and $|\cdot|$ is an extension of the absolute $|P/Q|:=q^{\deg(P)-\deg(Q)}$ to $K(\beta)$. In particular, such $\beta$ never satisfies the Roth Bound, i.e. $E(\beta)=2$, in odd characteristic. Nevertheless, any improvement of the Liouville bound $E(\beta)\leq d(\beta)$ often leads to nontrivial arithmetic consequences, and we exploit this principal here to prove our main result. 

The first such improvement in characteristic $p$ is due to Osgood \cite[Theorem 3]{Osgood}:
\begin{thm}[Osgood]{\label{Osgood}} Let $K=\mathbb{F}_q(t)$ and let $\beta\in\overline{K}\mysetminus \overline{\mathbb{F}}_q(t)$. Then   
\[ E(\beta)\leq \big\lfloor\frac{1}{2}(d(\beta)+3)\big\rfloor,\]      
unless $\beta\in K^{\sep}$ and $\beta$ satisfies:  
\begin{equation}{\label{eq:Ricatti}}
\beta'=a\beta^2+b\beta+c, \;\;\; \text{for some}\;\, a,b,c\in K;  
\end{equation} 
such a differential equation is called a generalized Riccati equation:
\end{thm} 
\begin{rmk} For $f\in K^{\sep}$, we define the derivative $f'\in K^{\sep}$ by extending the usual derivative $\frac{d}{dt}$ on $K$ via implicit differentiation.   
\end{rmk}    
Therefore, the technical condition (1) of Theorem \ref{thm:Silv} comes from our attempt to grapple with iterated preimages of infinity satisfying a generalized Riccati equation. To that end, we have the following fundamental lemma: 
\begin{lem}{\label{lem:Riccati}} Let $K=\mathbb{F}_q(t)$, and let $\phi(x)\in K(x)$ satisfy the following assumption:  
\begin{enumerate}[topsep=6pt, partopsep=6pt, itemsep=3pt]
\item[\textup{(1)}] The overdetermined system of linear equations corresponding to the augmented matrix $(M_\phi|R_\phi)$, with column vector $R_\phi=(r_n)$ defined by  
\[r_n=\sum_{i=0}^na_{d-i}b_{d-n+i}'-\sum_{i=0}^nb_{d-i}a_{d-n+i}'\]
for $0\leq n\leq6$, is inconsistent (i.e. has no solutions).     
\end{enumerate} 
If $\beta\in K^{\sep}$ is such that $\beta$ and $\phi(\beta)$ both satisfy a Riccati equation, then $[K(\beta):K]\leq 2d$.  
\end{lem} 
\begin{proof} Let $\phi(x)=F(x)/G(x)$ for some polynomials $F,G\in K[x]$ and suppose that $\beta\in K^{\sep}$ is such that
\begin{equation}{\label{diff1}}
\beta'=a\beta^2+b\beta+c\;\;\; \text{and}\;\;\; \phi(\beta)'=e\phi(\beta)^2+f\phi(\beta)+g
\end{equation} 
for some $a,b,c,e,f,g\in K$, i.e. assume that $\beta$ and $\phi(\beta)$ both satisfy a Riccati equation.

Note that we may assume that $G(\beta)\neq0$, since otherwise $[K(\beta):K]\leq d$ and we are finished. On the other hand, if $G(\beta)\neq0$ and $\beta\in K^{\sep}$, then we can differentiate the expression $F(\beta)=\phi(\beta)\cdot G(\beta)$ via the product rule to obtain 
\begin{equation}{\label{diff2}} 
F_1(\beta)+F_2(\beta)\cdot \beta'=\phi(\beta)\cdot\big[G_1(\beta)+G_2(\beta)\cdot\beta'\big]+G(\beta)\cdot\phi(\beta)'\,;
\end{equation}    
here, the polynomials $F_1$ and $F_2$, associated to $F(x)=a_dx^d+\dots +a_1x+a_0$, are defined by 
\[F_1(x)=a_d'x^d+\dots+a_1'x+ a_0'\,\;\;\;\;\text{and}\;\;\;\;\, F_2(x)=da_dx^{d-1}+(d-1)a_{d-1}x^{d-2}+\dots+ a_1,\]
and we define $G_1$ and $G_2$ (associated to $G$) in a similar manner. In particular, after substituting $\phi(\beta)=F(\beta)/G(\beta)$ and the differential equations in (\ref{diff1}) into (\ref{diff2}), we see that  \vspace*{.2cm}
\begin{gather}{\label{poly}}
\begin{split} 
\big[G(\beta)F_2(\beta)-F(\beta)G_2(\beta)\big]&(e\beta^2+f\beta+g)-F(\beta)G_1(\beta)+G(\beta)F_1(\beta)\\[5pt]
&- aF(\beta)^2-bF(\beta)G(\beta)-cG(\beta)^2=0,
\end{split} 
\end{gather}
and we obtain a polynomial $P_{\phi,\beta}\in K[x]$ that vanishes at $\beta$. One easily verifies that the $x^{2d+1}$ term above vanishes, so that $P_{\phi,\beta}$ is a polynomial of degree $\leq 2d$. Therefore, it suffices to show that $P_{\phi,\beta}$ is non-zero, to prove the lemma. To do this, we assume for a contradiction that $P_{\phi,\beta}$ is the zero polynomial. 

Now let us view the Riccati coefficients $(a,b,c,e,f,g)$ as indeterminates, and let $E_i(\phi)\in K[a,b,c,e,f,g]$ be the coefficient of $x^{2d-i}$ in the polynomial $P_{\phi}\in K[a,b,c,e,f,g,x]$ defined by (\ref{poly}) for $0\leq i\leq2d$; that is, $P_{\phi,\beta}\in K[x]$ is the polynomial in one variable obtained from $P_\phi$ via specializing the specific values $v_\beta:=(a,b,c,e,f,g)$ coming from the Riccati equations in (\ref{diff1}). In particular, each $E_i(\phi)$ is \emph{linear} in the variables $(a,b,c,e,f,g)$. Therefore, if $P_{\phi,\beta}$ is identically zero, we see that the vector $v_\beta\in K^6$ is a solution to the system of equations $E_i(\phi)=0$ for all $0\leq i\leq2d$.  

Of course, it is likely that for a generic $\phi$ of degree $d\geq4$, this system of equations has $\emph{no solutions}$, since the number of equations exceeds the number of variables. However, most $E_i(\phi)$ become quite complicated in large degree, and so in the interest of being as explicit as possible, we ignore all but the first seven equations, $E_i(\phi)=0$ for all $0\leq i\leq 6$. They are given as follows: \vspace{.15cm}
\begin{align*}
\begin{split} 
     E_0(\phi) :{}&  (-a_{d}^2)a + (-a_{d}b_{d})b + (-b_{d}^2)c + (a_{d}b_{d-1} - a_{d-1}b_{d})e =r_0. 
\end{split}\\[7pt]
\begin{split}
    E_1(\phi) :{}&(-2a_{d}a_{d-1})a + (-a_{d}b_{d-1} - a_{d-1}b_{d})b + (-2b_{d}b_{d-1})c + (2a_{d}b_{d-2} - 2a_{d-2}b_{d})e \\
&+ (a_{d}b_{d-1} - a_{d-1}b_{d})f =r_1.
\end{split}\\[7pt]
\begin{split} 
     E_2(\phi) :{} &(-2a_{d}a_{d-2} - a_{d-1}^2)a + (-a_{d}b_{d-2} - a_{d-1}b_{d-1} - a_{d-2}b_{d})b\\
&+ (-2b_{d}b_{d-2} - b_{d-1}^2)c + (3a_{d}b_{d-3} + a_{d-1}b_{d-2} - a_{d-2}b_{d-1} - 3a_{d-3}b_{d})e\\ 
&+ (2a_{d}b_{d-2} - 2a_{d-2}b_{d})f  + (a_{d}b_{d-1} - a_{d-1}b_{d})g= r_2.
\end{split}\\[7pt]
\begin{split}
     E_3(\phi) :{}& (-2a_da_{d-3} - 2a_{d-1}a_{d-2})a + (-a_db_{d-3} - a_{d-1}b_{d-2} - a_{d-2}b_{d-1} - a_{d-3}b_d)b \\ 
     & + (-2b_db_{d-3} -2b_{d-1}b_{d-2})c+ (4a_db_{d-4} + 2a_{d-1}b_{d-3} - 2a_{d-3}b_{d-1} - 4a_{d-4}b_d)e\\ 
     & + (3a_db_{d-3} + a_{d-1}b_{d-2} -a_{d-2}b_{d-1} - 3a_{d-3}b_d)f + (2a_db_{d-2} - 2a_{d-2}b_d)g = r_3
\end{split}\\[7pt]
\begin{split}
     E_4(\phi) :{}& (-2a_da_{d-4} - 2a_{d-1}a_{d-3} - a_{d-2}^2)a + (-a_db_{d-4} - a_{d-1}b_{d-3} - a_{d-2}b_{d-2} - a_{d-3}b_{d-1}\\ 
     & - a_{d-4}b_d)b + (-2b_db_{d-4} - 2b_{d-1}b_{d-3} - b_{d-2}^2)c + (5a_db_{d-5} + 3a_{d-1}b_{d-4} + a_{d-2}b_{d-3}\\ 
     & - a_{d-3}b_{d-2}- 3a_{d-4}b_{d-1} - 5a_{d-5}b_d)e + (4a_db_{d-4} + 2a_{d-1}b_{d-3} - 2a_{d-3}b_{d-1} - 4a_{d-4}b_d)f \\ 
     & + (3a_db_{d-3} + a_{d-1}b_{d-2} - a_{d-2}b_{d-1} - 3a_{d-3}b_d)g= r_4.    
\end{split}\\[7pt]
\begin{split}
     E_5(\phi) :{}& (-2a_da_{d-5} - 2a_{d-1}a_{d-4} - 2a_{d-2}a_{d-3})a + (-a_db_{d-5} - a_{d-1}b_{d-4} - a_{d-2}b_{d-3} - a_{d-3}b_{d-2}\\ 
     &- a_{d-4}b_{d-1} -a_{d-5}b_d)b + (-2b_db_{d-5} - 2b_{d-1}b_{d-4} - 2b_{d-2}b_{d-3})c + (6a_db_{d-6} + 4a_{d-1}b_{d-5} \\
     & + 2a_{d-2}b_{d-4} - 2a_{d-4}b_{d-2} - 4a_{d-5}b_{d-1} - 6a_{d-6}b_d)e + (5a_db_{d-5} + 3a_{d-1}b_{d-4} + a_{d-2}b_{d-3}\\ 
     & - a_{d-3}b_{d-2} - 3a_{d-4}b_{d-1} - 5a_{d-5}b_d)f + (4a_db_{d-4} + 2a_{d-1}b_{d-3} - 2a_{d-3}b_{d-1} - 4a_{d-4}b_d)g\\ 
     & = r_5.  
\end{split}\\[7pt]
\begin{split}
E_6(\phi) :{} & (-2a_da_{d-6} - 2a_{d-1}a_{d-5} - 2a_{d-2}a_{d-4} - a_{d-3}^2)a + (-a_{d}b_{d-6} - a_{d-1}b_{d-5} - a_{d-2}b_{d-4}\\ 
      & - a_{d-3}b_{d-3} - a_{d-4}b_{d-2} - a_{d-5}b_{d-1} - a_{d-6}b_d)b + (-2b_db_{d-6} - 2b_{d-1}b_{d-5} - 2b_{d-2}b_{d-4}\\ 
      & - b_{d-3}^2)c + (7a_db_{d-7} + 5a_{d-1}b_{d-6} + 3a_{d-2}b_{d-5} + a_{d-3}b_{d-4} - a_{d-4}b_{d-3} - 3a_{d-5}b_{d-2}\\ 
      &- 5a_{d-6}b_{d-1}- 7a_{d-7}b_d)e+ (6a_db_{d-6}+ 4a_{d-1}b_{d-5} + 2a_{d-2}b_{d-4} - 2a_{d-4}b_{d-2} - 4a_{d-5}b_{d-1}\\ 
      &- 6a_{d-6}b_d)f+ (5a_db_{d-5} + 3a_{d-1}b_{d-4} + a_{d-2}b_{d-3}- a_{d-3}b_{d-2} - 3a_{d-4}b_{d-1} - 5a_{d-5}b_d)g= r_6.  
\end{split}      
\end{align*}
Here the $r_n\in K$ for $0\leq n\leq 6$ are defined as in Lemma \ref{lem:Riccati}. In particular, since this system of equations, corresponding to the augmented matrix $(M_\phi| R_\phi)$, is assumed to be inconsistent, there can be no solutions $(a,b,c,e,f,g)\in K^6$. Hence, the polynomial $P_{\phi,\beta}\in K[x]$ is non-zero and $[K(\beta):K]\leq2d$ as claimed.         
\end{proof}   
Now let $K_v$ be a completion of the absolute value $|\cdot|$ on $K$ corresponding to the valuation $v=\ord_\infty$ on $K$. We can associate the \emph{chordal metric} $\rho:\mathbb{P}^1(\overline{K}_v)\times\mathbb{P}^1(\overline{K}_v) \rightarrow [0,1]$ given by 
\[\rho\big([X_1,Y_1], [X_2,Y_2]\big)=\frac{|X_1Y_2-X_2Y_1|}{\max\{|X_1|,|Y_1|\}\, \max\{|X_2|,|Y_2|\}}\] 
A nice property that we use in the proof of Theorem \ref{thm:Silv} is that rational maps are Lipschitz with respect to the chordal metric: 
\begin{lem}{\label{lem:lip}} Let $\phi(x)\in K(x)$ be a rational map. Then $\phi$ is Lipschitz with respect to the chordal metric. Specifically,  
\[\rho(\phi(P_1),\phi(P_2))\leq |\Res(\phi)|^{-2}\,\rho(P_1,P_2)\;\;\;\;\; \text{for all}\;\, P_1,P_2\in\mathbb{P}^1(\overline{K}_v);\]
here $\Res(\phi)$ is the resultant of $\phi$; see \cite[\S2.4]{Silv-Dyn} for a definition.    
\end{lem} 
However, as in the proof of Silverman's limit theorem over number fields \cite{Silv-Int}, we also need to control the chordal metric under taking preimages of rational functions; see Lemma 3.51 and Excercise 4.43 of \cite{Silv-Dyn} for a proof of the following fact. 
\begin{lem}\label{lem:inv} Let $\phi(x)\in K(x)$ be a map of degree $d\geq2$. For $Q\in\mathbb{P}^1(\overline{K}_v)$, let   
\[\textbf{e}_Q(\phi):=\max_{Q'\in\phi^{-1}(Q)} e_{Q'}(\phi)\]
be the maximum of the ramification indices of the points in the inverse image of $Q$. Then there is a positive constant $C_v=C_v(\phi,Q)$ such that 
\[ \min_{Q'\in\phi^{-1}(Q)}\rho(P,Q')^{\textbf{e}_Q(\phi)}\leq \frac{1}{C_v}\,\rho(\phi(P),Q)\;\;\;\;\text{for all}\;\,P\in\mathbb{P}^1(\overline{K}_v).\]\\ 
In other words, if $\phi(P)$ is close to $Q$, then there is a point in the inverse image of $Q$ that is close to $P$, but ramification affects how close. 
\end{lem} 
Finally, we need an elementary lemma that compares the chordal metric $\rho(x,y)$ to the usual distance $|x-y|$ for certain $x,y\in\overline{K}_v$; see \cite[Lemma 3.53]{Silv-Dyn} for a proof of the following:  
\begin{lem}{\label{lem:comp}} Let $\rho$ be the chordal metric on $\mathbb{P}^1(\overline{K}_v)$ and suppose that $x,y\in\overline{K}_v$ satisfy $\rho(x,y)\leq\frac{1}{2}\,\rho(y,\infty)$. Then \[\rho(x,y)\geq|x-y|\cdot \frac{\rho(y,\infty)^2}{2}.\]  
\end{lem}  
With these auxiliary estimates in place, we are now ready to prove our main result. 
\section{Main Argument}{\label{sec:Main}} 
\begin{proof}[(Proof of Theorem \ref{thm:Silv})] To prove statement (\ref{silvlim}), it suffices to show that for all $\epsilon>0$ there exist only finitely many iterates $n$ satisfying $\deg(a_n)\geq (2+\epsilon)\deg(b_n)$. Equivalently, we wish to show that the set 
\begin{equation}{\label{set}}
N(\phi,\alpha, \epsilon)=\big\{n\,:\; |a_n|\geq |b_n|^{2+\epsilon}\big\}\vspace{.1cm}
\end{equation}  
is finite. To do this, suppose that $0< \epsilon\leq\frac{1}{5}$ and that $n\in N(\phi,\alpha,\epsilon)$. Then for such $\epsilon$, 
\begin{equation}{\label{ineq1}} 
\frac{1}{|a_n|^{\frac{1}{2}+\epsilon^2}}\geq \frac{1}{|a_n|^{\frac{1+\epsilon}{2+\epsilon}}}=\frac{\;\;\;\,|a_n|^{\frac{1}{2+\epsilon}}}{|a_n|\;}\geq\frac{|b_n|}{|a_n|}=\rho(\phi^n(\alpha),\infty).\\
\end{equation} 
Now let us assume, for a contradiction, that $N(\phi,\alpha,\epsilon)$ is infinite. Lemma \ref{lem:inv} implies that for all $m\geq1$ and $n>m$ there exists $\beta_m\in\mathbb{P}^1(\overline{K}_v)$ such that:
\begin{equation}{\label{ineq2}}
\phi^m(\beta_m)=\infty\;\;\;\;\;\;\;\;\;\text{and}\;\;\;\;\;\;\;\;\;\; \rho(\phi^n(\alpha),\infty)\geq C(\phi,m)\cdot \rho(\phi^{n-m}(\alpha), \beta_m)
\end{equation} 
for some constant $C(\phi,m)=C(\phi^m,\infty)$; here we use crucially that $\infty\not\in\PostCrit_\phi$, which implies that the ramification indices 
\[e_\beta(\phi^m)=e_\beta(\phi) e_{\phi(\beta)}(\phi)\dots \,e_{\phi^{m-1}(\beta)}(\phi)=1\]
for all $\beta\in \mathbb{P}^1(\overline{K}_v)$ satisfying $\phi^m(\beta)=\infty$. To see this, note that if $e_\beta(\phi^m)>1$, then the product above implies that $\phi^{j}(\beta)=\gamma_j$ is a ramification point of $\phi$ for some $0\leq j<m$; hence, 
\[\phi^{m-j}(\gamma_j)=\phi^{m-j}(\phi^{j}(\beta))=\phi^m(\beta)=\infty,\] 
which contradicts assumption (2) of Theorem \ref{thm:Silv}. In particular, if $m>0$, then (\ref{ineq1}) and (\ref{ineq2}) together imply that 
\begin{equation}{\label{ineq3}}
\frac{1}{|a_n|^{\frac{1}{2}+\epsilon^2}}\geq C(\phi,m)\cdot\rho(\phi^{n-m}(\alpha), \beta_m)
\end{equation}  
\\
for all $n>m$ and $n\in N(\phi,\alpha,\epsilon)$. Now fix, once and for all, an integer 
\begin{equation}{\label{m}}
m\geq\max\Big\{\log_d\Big(\frac{3}{2\epsilon^2}\Big)+1,4\Big\}.
\end{equation}  
Since $\alpha$ is wandering, it follows that $\{\deg({a_n})\}_{n\in N(\phi,\alpha,\epsilon)}$ is an unbounded sequence: otherwise, $h(\phi^n(\alpha))=\deg(a_n)$ is bounded over all $n\in N(\phi,\alpha,\epsilon)$, and since there are only finitely many points of $K$ of bounded height by Northcott's theorem \cite[Theorem 3.7]{Silv-Dyn}, we must have that $\phi^{n_1}(\alpha)=\phi^{n_2}(\alpha)$ for some $n_1\neq n_2$, a contradiction. Therefore, (\ref{ineq3}) implies that $\rho(\phi^{n-m}(\alpha), \beta_m)\rightarrow0$ as $n$ grows, and so we may choose $n$ sufficiently large so that 
\[\rho(\phi^{n-m}(\alpha), \beta_m)\leq \frac{1}{2}\, \rho(\beta_m,\infty).\] 
Therefore, if $\beta_m\neq\infty$, then Lemma \ref{lem:comp} and (\ref{ineq3}) imply that 
\begin{equation}{\label{ineq4}}
 \frac{1}{|a_n|^{\frac{1}{2}+\epsilon^2}}\geq C(\phi,m)\cdot \big|\phi^{n-m}(\alpha)-\beta_m\big| \cdot \frac{\rho(\beta_m,\infty)^2}{2}.
\end{equation} 
We now proceed in cases: in Cases (1)-(3), we assume that $\beta_m$ is not the point at infinity. 
\vspace{.1in}
\\
\textbf{Case (1):} Suppose that $\beta_m\in\overline{\mathbb{F}}_q(t)$ and write $\beta_m=p/q$ for some polynomials $p, q \in \overline{\mathbb{F}}_q[t]$. Since $\alpha$ is wandering and $m$ is fixed, we may choose $n$ sufficiently large so that $\phi^{n-m}(\alpha)\neq \beta_m$. In particular, (\ref{ineq4}) reduces to 
\[ \frac{1}{|a_n|^{\frac{1}{2}+\epsilon^2}}\geq C(\phi,m)\cdot \bigg|\frac{a_{n-m}q-b_{n-m}p}{qb_{n-m}}\bigg| \cdot \frac{\rho(\beta_m,\infty)^2}{2}\geq \kappa_1(\phi,m)\cdot\frac{1}{|b_{n-m}|}\]
for some constant $\kappa_1(\phi,m)$; here we use the trivial lower bound $|a_{n-m}q-b_{n-m}p|\geq1$, which holds since $a_{n-m}q-b_{n-m}p$ is a nonzero polynomial. After applying $\log_q(\cdot)$ to both sides of the inequality above, we see that  
\[(1/2+\epsilon^2)\deg(a_n) \leq \deg(b_{n-m})+\log_q(\kappa_1(\phi,m))\leq h(\phi^{n-m}(\alpha))+\log_q(\kappa_1(\phi,m)).\]
On the other hand, \[\deg(a_n)=\max\{\deg(a_n),\deg(b_n)\}=h(\phi^n(\alpha))\;\;\;\;\;\;\text{for all}\;\;n\in N(\phi,\alpha,\epsilon),\]
and it follows from standard properties of the canonical height \cite[Theorem 3.20]{Silv-Dyn}: 
\[h=\hat{h}_\phi +O(1)\;\;\; \text{and}\;\;\; \hat{h}_\phi(\phi^k(P))=d^k\cdot\hat{h}_\phi(P)\] 
for all $P\in\mathbb{P}^1(\overline{K})$, that 
\[d^n\cdot \hat{h}_\phi(\alpha)\leq d^{n-m+1}\cdot \hat{h}_\phi(\alpha)+\kappa_2(\phi,m,\epsilon);\] 
here we use also that $d\geq2$ and $\epsilon$ is positive. In particular, since $\hat{h}_\phi(\alpha)\neq0$, we may divide this term out and obtain 
\[d^n\leq d^{n-m+1}+\kappa_3(\phi,\alpha, m,\epsilon)\leq d^{n-3}+\kappa_3(\phi,\alpha, m,\epsilon),\] 
since $m$ is at least $4$. Therefore, $d^3\geq 8$ implies that 
\[n\leq \log_d\bigg(\frac{\kappa_3(\phi,\alpha,m,\epsilon)}{7}\bigg)+3,\] 
and $N(\phi,\alpha,\epsilon)$ is finite as claimed.  \vspace{.075in}
\\ 
\textbf{Case (2):} Suppose that $\beta_m\not\in\overline{\mathbb{F}}_q(t)$ and $[K(\beta_m):K]\leq 2d$. Then the ordinary Liouville bound for $\beta_m$ established by Mahler \cite{Mahler} and the upper bound in (\ref{ineq4}) imply that
\begin{align*}
(1/2+\epsilon^2)\deg(a_n) &\leq 2d\cdot\deg(b_{n-m})+\log_q\Big( C(\phi,m)\,\frac{\rho(\beta_m,\infty)^2}{2}\,\Big)\\[5pt]
&\leq 2d\cdot h(\phi^{n-m}(\alpha))+\kappa_4(\phi,m)
\end{align*}
for some constant $\kappa_4(\phi,m)$. On the other hand, $\deg(a_n)=h(\phi^n(\alpha))$ for all $n\in N(\phi,\alpha, \epsilon)$, and it follows from standard properties of the canonical height \cite[Theorem 3.20]{Silv-Dyn} that  
\[(1/2+\epsilon^2)\cdot\hat{h}_\phi(\alpha)\cdot d^n\leq d^{n-m+2}\cdot \hat{h}_\phi(\alpha)+\kappa_5(\phi,m,\epsilon).\] 
Moreover, since $\hat{h}_\phi(\alpha)\neq0$, $d\geq2$ and $\epsilon>0$, we can deduce that   
\[d^n\leq d^{n-m+3}+\kappa_6(\phi,\alpha, m,\epsilon).\] 
On the other hand, $m\geq4$ by assumption, so that the bound above implies that 
\[n\leq\log_d(\kappa_6(\phi,\alpha, m,\epsilon))+1.\]
Therefore, $N(\phi,\alpha,\epsilon)$ is finite as claimed.\vspace{.075in}
\\ 
\indent \textbf{Case (3):} Now suppose that $\beta_m\not\in\overline{\mathbb{F}}_q(t)$ and $[K(\beta_m):K]> 2d$. Then Lemma \ref{lem:Riccati} and assumption (1) of Theorem \ref{thm:Silv} imply that $\beta_m$ and $\phi(\beta_m)$ cannot both satisfy a Riccati equation (whenever $\beta_m$ is separable).  

We assume first that $\beta_m$ does not satisfy such an equation. Then (\ref{ineq4}) and \cite[Theorem 3]{Osgood} applied to $\beta_m$ (separable or inseparable) imply that 
\begin{align*}
(1/2+\epsilon^2)\deg(a_n) \leq& \Big(\frac{d(\beta_m)+3}{2}\Big)\deg(b_{n-m})+\log_q\Big( C(\phi,m)\,\frac{\rho(\beta_m,\infty)^2}{2}\,\Big)+\kappa_7(\phi,m)\\[8pt]
\leq & \Big(\frac{d^m+3}{2}\Big)h(\phi^{n-m}(\alpha)) + \kappa_8(\phi,m);
\end{align*} 
recall that $d(\beta_m)=[K(\beta_m):K]$ is the algebraic degree of $\beta_m$ over $K$, and hence $d(\beta_m)\leq d^m$. On the other hand, $h(\phi^n(\alpha))=\deg(a_n)$ for all $n\in N(\phi,\alpha,\epsilon)$, so that 
\[ (1/2+\epsilon^2)\cdot \hat{h}_\phi(\alpha)\cdot d^n \leq \Big(\frac{d^m+3}{2}\Big)\cdot d^{n-m}\cdot \hat{h}_\phi(\alpha) + \kappa_9(\phi,m,\epsilon) \]  
follows from properties of the canonical height \cite[Theorem 3.20]{Silv-Dyn}. Since $\hat{h}_\phi(\alpha)\neq0$, we may divide this term out and simplify to 
\[d^n\leq \frac{3}{2\epsilon^2}\,d^{n-m}+ \kappa_{10}(\phi,\alpha,m,\epsilon)\leq d^{n-1}+\kappa_{10}(\phi,\alpha,m,\epsilon);\]
here we use our assumption that $m\geq\max\big\{\log_d\big(\frac{3}{2\epsilon^2}\big)+1,4\big\}$. In particular, we deduce that 
\[n\leq\log_d\big(\kappa_{10}(\phi,\alpha, m,\epsilon)\big)+1,\] 
and that $N(\phi,\alpha,\epsilon)$ is again finite.

Likewise, if $\phi(\beta_m)$ does not satisfy a Riccati equation, then the fact that rational maps are Lipschitz with respect to the chordal metric (see Lemma \ref{lem:lip} above), together with (\ref{ineq3}), imply that
\[ \frac{1}{|a_n|^{\frac{1}{2}+\epsilon^2}}\geq C(\phi,m)\cdot\Res(\phi)^2\cdot\rho\big(\phi^{n-m+1}(\alpha),\, \phi(\beta_m)\big).\]
On the other hand, Lemma \ref{lem:comp} and the fact that $\deg(a_n)\rightarrow \infty$ for increasing $n\in N(\phi,\alpha,\epsilon)$ imply that we may choose $n$ sufficiently large so that 
\begin{equation}{\label{ineq5}}
\frac{1}{|a_n|^{\frac{1}{2}+\epsilon^2}}\geq C(\phi,m)\cdot\Res(\phi)^2\cdot \big|\phi^{n-m+1}(\alpha)-\phi(\beta_m)\big|\cdot \frac{\rho(\phi(\beta_m),\infty)^2}{2}.
\end{equation} 
From here, the argument proceeds as in the first part of Case (3). Specifically, (\ref{ineq5}) and the Osgood bound \cite[Theorem 3]{Osgood} applied to the algebraic function $\phi(\beta_m)$ yield \vspace{.03cm}:  
\begin{align*}
\hspace*{-.75cm} 
(1/2+\epsilon^2)\deg(a_n) \leq& \bigg(\frac{d(\phi(\beta_m))+3}{2}\bigg)\deg(b_{n-m+1})+\log_q\bigg( C(\phi,m)\,\Res(\phi)^2\,\frac{\rho(\phi(\beta_m),\infty)^2}{2}\,\bigg)+\kappa_{11}(\phi,m)\\[8pt]
\leq & \Big(\frac{d^{m-1}+3}{2}\Big)h(\phi^{n-m+1}(\alpha)) + \kappa_{12}(\phi,m); \vspace{.01cm} 
\end{align*} 
recall that $d(\phi(\beta_m))$ is the degree of $\phi(\beta_m)$ over $K$, and note that we have the trivial bound $d(\phi(\beta_m))\leq d^{m-1}$, since 
\[\phi^{m-1}(\phi(\beta_m))=\phi^m(\beta_m)=\infty\] 
and $\phi^{m-1}$ is a map of degree $d^{m-1}$. On the other hand, $h(\phi^n(\alpha))=\deg(a_n)$ for $n\in N(\phi,\alpha,\epsilon)$, so that 
\[ (1/2+\epsilon^2)\cdot \hat{h}_\phi(\alpha)\cdot d^n \leq \Big(\frac{d^{m-1}+3}{2}\Big)\cdot d^{n-m+1}\cdot \hat{h}_\phi(\alpha) + \kappa_{13}(\phi,m,\epsilon) \]  
follows from properties of the canonical height \cite[Theorem 3.20]{Silv-Dyn}. Finally, since $\hat{h}_\phi(\alpha)\neq0$, we may divide this term out and simplify to 
\[d^n\leq \frac{3}{2\epsilon^2}\,d^{n-m}+ \kappa_{14}(\phi,\alpha,m,\epsilon)\leq d^{n-1}+\kappa_{14}(\phi,\alpha,m,\epsilon);\]
here again we use our assumption that $m\geq\max\big\{\log_d\big(\frac{3}{2\epsilon^2}\big)+1,4\big\}$. In particular, we deduce that \[n\leq\log_d\big(\kappa_{14}(\phi,\alpha, m,\epsilon)\big)+1,\] 
and that $N(\phi,\alpha,\epsilon)$ is finite.     
\vspace{.075in}
\\ 
\indent \textbf{Case (4):} Finally, suppose that $\beta_m$ is the point at infinity; in particular, $\phi^m(\infty)=\infty$ and $\infty$ is a periodic point of $\phi$. Then (\ref{ineq3}) implies that  
\[\frac{1}{|a_n|^{\frac{1}{2}+\epsilon^2}}\geq C(\phi,m)\cdot \frac{|b_{n-m}|}{\max\big\{|a_{n-m}|\,,|b_{n-m}|\big\}}\geq C(\phi,m)\cdot\frac{1}{H(\phi^{n-m}(\alpha))};\]
here $H(P)=q^{\,h(P)}$ is the multiplicative Weil height of $P\in \mathbb{P}^1(\overline{K})$. Note that  we may assume that $b_{n-m}\neq0$ for all $n$ sufficiently large, since $\alpha$ is wandering. In particular, we see that 
\[ (1/2+\epsilon^2)h(\phi^n(\alpha))=(1/2+\epsilon^2)\deg(a_n)\leq h(\phi^{n-m}(\alpha))+ \log_q(C_m).\] 
It follows from standard properties of canonical heights and that $d\geq2$, that  
\[d^n\leq d^{n-m+1}+\kappa_{15}(\phi,\alpha,m,\epsilon)\leq d^{n-3}+\kappa_{15}(\phi,\alpha,m,\epsilon);\]
here we use our assumption that $m$ is at least $4$. Therefore, 
\[n\leq \log_d\bigg(\frac{\kappa_{15}(\phi,\alpha,m,\epsilon)}{7}\bigg)+3,\] 
and $N(\phi,\alpha,\epsilon)$ is finite; note that in this case, we have proven the $\limsup$ part of Theorem \ref{thm:Silv} without the use of the Mahler \cite{Mahler} or Osgood \cite{Osgood} bound on the diophantine exponent. Finally, let $\pi(x):=1/\phi(1/x)$ and note that $\pi^n(1/\alpha)=b_n/a_n$. In particular, by repeating the argument above for $\pi$, we obtain the lower bound in Theorem \ref{thm:Silv}.                       
\end{proof} 
\begin{rmk} It is likely that one can improve the limit bound in Theorem \ref{thm:Silv} to the optimum bound of $1$ in Case (1), Case (2) and Case (4) above. However, these cases are extremely rare (in fact, it is not even clear to the author that Case (1) and Case (2) can occur for all $m$ sufficiently large, given the assumptions in Theorem \ref{thm:Silv} and conjectures related to eventual stability \cite{Rafe}), and we have chosen to keep our arguments consistent across all cases. 
\end{rmk}
We conclude this section with a family of rational maps in every degree $d\geq6$ satisfying the hypothesis of Theorem \ref{thm:Silv}. For a fixed degree $d$, this family likely has moderate codimension in $\Rat_d$, the space of all rational maps of degree $d$; for more on $\Rat_d$, see \cite[\S4.3]{Silv-Dyn} 
\begin{prop}{\label{prop:example}} Let $K=\mathbb{F}_q(t)$ have characteristic $p>5$, and let $d\geq6$. If $h(x)\in K[x]$ is a polynomial of degree at most $d-6$ such that $x^d+t^2x^{d-1}+tx^{d-2}+tx^{d-5}+h(x)$ is irreducible over $K$, then
\begin{equation}{\label{eg1}}
\phi(x):=\frac{x^{d}}{x^d+t^2x^{d-1}+tx^{d-2}+tx^{d-5}+h(x)}
\end{equation} 
satisfies hypothesis \textup{(1)} and \textup{(2)} of Theorem \ref{thm:Silv}. In particular, this statement holds for $h(x)=t$ by Eisenstein's criterion.    
\end{prop} 
\begin{proof} Let $M_{\phi}'$ and $R_\phi'$ be the matrices obtained from $M_\phi$ and $R_\phi$ by deleting their bottom rows. Since $M_\phi'$ has determinant $-12t^6\neq0$, we see that the subsystem of linear equations $(M_\phi',R_\phi')$ has a unique solution 
\[ (a,b,c,e,f,g)=\bigg(\frac{1/2t^3 + 7/2}{t},\, \frac{-6}{t},\, \frac{5}{2t},\, \frac{1}{2},\, 0,\, \frac{5}{2}t\bigg).\]
On the other hand, after substituting this solution into the deleted equation corresponding to the bottom row of $(M_\phi,R_\phi)$, we see that $\frac{15}{2}t^2=0$, a contradiction. In particular, the system $(M_\phi,R_\phi)$ has no solutions and assumption (1) of Theorem \ref{thm:Silv} holds.  

As for ramification, since $\phi$ is not defined over $\mathbb{F}_q(t^p)$, it follows from \cite[Exercise 1.10]{Silv-Dyn} that $\phi$ is a separable. Likewise, since the denominator $g(x)=x^d+t^2x^{d-1}+tx^{d-2}+tx^{d-5}+h(x)$ is not defined over $\mathbb{F}_q[t^p]$ and irreducible, the poles of $\phi$ are not critical points. Similarly, $\infty$ is not a critical point of $\phi$ by \cite[Exercise 1.6(b)]{Silv-Dyn}. Therefore, any critical points of $\phi$ must be a root of $x^{d-1}\big(t^2x^{d-1}+2tx^{d-2}+5tx^{d-5}+d\cdot h(x)-xh'(x)\big)$, the numerator of the formal derivative of (\ref{eg1}). In particular, if $\gamma$ is a critical point of $\phi$, then $[K(\gamma):K]\leq d-1$.

Now suppose that $\phi^n(\gamma)=\infty$ for some $n\geq1$. Then $\phi(\phi^{n-1}(\gamma))=\infty$, and $\phi^{n-1}(\gamma)$ must be a root of $g$. However, $g$ is irreducible over $K$, and hence the smallest degree extension of $K$ over which $g$ has a root is $d$. But $\phi^{n-1}(\gamma)\in K(\gamma)$, and thus 
\[[K(\phi^{n-1}(\gamma)):K]\leq [K(\gamma):K]\leq d-1,\] 
a contradiction. Therefore, $\infty\not\in\PostCrit_\phi$ and assumption (2) of Theorem \ref{thm:Silv} holds.                    
\end{proof}  
\section{Latt\'{e}s Maps and Elliptic Curves}{\label{sec:lattes}}  
Unfortunately, Theorem \ref{thm:Silv} does not apply directly to Latt\'{e}s maps, i.e. the rational maps on the projective line associated to endomorphisms on elliptic curves \cite[\S 6.4]{Silv-Dyn}. The reason is that infinity is usually contained in the post-critical orbit of these maps; see \cite[Lemma 6.38]{Silv-Dyn} and \cite[Proposition 6.45]{Silv-Dyn}. However, our methods can be adapted to Latt\'{e}s maps of degree $4$, an we use this to study canonical heights on elliptic curves \cite[VIII.9]{Silv-Ell}.  
\begin{thm}{\label{thm:lattes}} Let $K=\mathbb{F}_q(t)$ have characteristic at least $5$, let $E: y^2=x^3+Ax+B$ be an elliptic curve over $K$, and let 
\[\phi_{E,2}(x)= \frac{x^4-2Ax^2-8Bx+A^2}{4x^3+4Ax+4B}\]
be the associated Latt\'{e}s map of degree $4$ on $\mathbb{P}^1$ corresponding to multiplication by $2$ on $E$. For non-torsion points $P\in E(K)$, write 
\[x\big([2^n]P\big)=\phi_{E,2}^n(x(P))=\frac{a_n(t)}{b_n(t)}\]
for some polynomials $a_n(t), b_n(t)\in\mathbb{F}_q[t]$ in lowest terms. If $2AB'-3BA'\neq0$, then \vspace{.2cm} 
\begin{align*} \frac{1}{2}\cdot \liminf_{n\rightarrow\infty}\frac{\deg(a_n)}{4^n}\leq\;&\hat{h}_E(P)\leq2\cdot \limsup_{n\rightarrow\infty}\frac{\deg(a_n)}{4^n}\\[5pt]  
&\;\;\,\textup{and}\\[5pt] 
\frac{1}{2}\cdot \liminf_{n\rightarrow\infty}\frac{\deg(b_n)}{4^n}\leq\;&\hat{h}_E(P)\leq2\cdot \limsup_{n\rightarrow\infty}\frac{\deg(b_n)}{4^n}. 
\end{align*}
In other words, the canonical height of a non-torsion point on $E$ can be well approximated by the degree of the numerator (or denominator) of a large multiple of the point.       
\end{thm} 
\begin{proof}[(Proof of Theorem \ref{thm:lattes})] To avoid overly cumbersome notation, we simply write $\phi$ for the Latt\'{e}s map corresponding to multiplication by $2$. To prove Theorem \ref{thm:lattes}, we proceed as in the proof of Theorem \ref{thm:Silv}. However, there are two differences. First, the point at infinity is always in the post-critical orbit of $\phi$. Therefore, we must show that the ramification index $e_{\phi^m}(\beta)$ for any $\beta\in\overline{K}$ satisfying $\phi^m(\beta)=\infty$, is not too large. 
\begin{lem}{\label{ram:lattes}} Let $E$ be an elliptic curve and let $\phi_{E,n}=\phi$ be a Latt\'{e}s map associated to multiplication by $n$ on $E$. Fix $Q\in\mathbb{P}^1(\overline{K})$, and suppose that $\beta\in\mathbb{P}^1(\overline{K})$ satisfies $\phi^m(\beta)=Q$ for some iterate $m\geq1$. Then $e_{\phi^m}(\beta)\leq2$.  
\end{lem}
\begin{proof}[(Proof of Lemma \ref{ram:lattes})] It is known that $\phi$ has the following properties: every critical point of $\phi$ is simple (ramification degree $2$) and $\phi$ has exactly four post-critical points, non of which is also critical; In fact, in characteristic zero, these properties completely characterize the Latt\'{e}s maps \cite[P. 247, Remark 3.6 ]{Milnor}. Now suppose that $\phi^{n_1}(\beta)=\gamma_1$ and  $\phi^{n_2}(\beta)=\gamma_2 $ for some integers $0\leq n_1<n_2\leq m-1$ and some critical points $\gamma_1$ and $\gamma_2$ of $\phi$. Then \[\phi^{n_2-n_1}(\gamma_1)=\phi^{n_2-n_1}(\phi^{n_1}(\beta))=\phi^{n_2}(\beta)=\gamma_2,\] a contradiction of the fact that no critical points of $\phi$ are post-critical. On the other hand, ramification indices are multiplicative: \[e_\beta(\phi^m)=e_\beta(\phi) e_{\phi(\beta)}(\phi)\dots \,e_{\phi^{m-1}(\beta)}(\phi).\] Therefore, at most one term in the product above is larger than $1$, and if such a term exists, then it must be equal to $2$ (since every critical point of $\phi$ is simple).        
\end{proof}  
The second difference between the proof of Theorem \ref{thm:Silv} and the proof of Theorem \ref{thm:lattes} is that we must adapt Lemma \ref{lem:Riccati} to include information about $\phi$ and its second iterate. The reason is that the system of linear equations $(M_{\phi},R_{\phi})$ has a unique solution in this case (essentially because $\phi$ has small degree and relatively non-generic coefficients). However, we can remedy this by passing to a larger iterate.
\begin{lem}{\label{riccati:lattes}} Let $E$ and $\phi$ be as in Theorem \ref{thm:lattes} and assume that $2AB'-3BA'\neq0$. If $\beta\in K^{\sep}$ is such that $\beta$, $\phi(\beta)$ and $\phi^2(\beta)$ all satisfy a Riccati equation, then $[K(\beta):K]\leq 32$.  
\end{lem} 
\begin{proof}[(Proof of Lemma \ref{riccati:lattes})] Suppose that $\beta$, $\phi(\beta)$ and $\phi^2(\beta)$ all satisfy a Riccati equation, and that $[K(\beta):K]>32$. In particular, since $\beta$ and $\phi(\beta)$ both satisfy a Riccati equation, the proof of Lemma \ref{lem:Riccati} implies that the polynomial $P_{\phi,\beta}$ of degree  at most $8$ defined in (\ref{poly}) must be identically zero. Hence, the system of equations $(M_\phi, R_\phi)$ must have a solution. On the other hand, the submatrix $M_E^{(1)}$ obtained from $M_\phi$ by deleting the bottom row is given by 
\begin{equation*} M_E^{(1)}:=
\begin{pmatrix}
    -1&0 & 0&4&0&0\\
   0&-4&0&0&4&0\\
   4A&0 &-16 &20A &0 &4\\ 
   16B &4A &0 & 80B &20A &0\\
   -6A^2& 28B& 32A&-20A^2& 80B&20A\\
   -32AB& 4A^2& -32B&-16AB &-20A^2& 80B
\end{pmatrix} 
\end{equation*} 
and we compute that the determinant of $M_E^{(1)}$ is $2^{15}3^{4}B(4A^3+27B^2)\neq0$; note that this quantity is nonzero since $E$ is non-singular, $K$ does not have characteristic $2$ or $3$, and $2AB'-3BA'\neq0$. In particular, there are unique coefficients $(a,b,c,e,f,g)\in K^6$, defined by (\ref{diff1}), that force the polynomial $P_{\phi,\beta}$ to be identically zero. Specifically, $\beta$ must satisfy the differential equation 
\begin{equation}{\label{diff:lattes1}}
\beta'=\frac{(-6AB' + 9BA')}{(4A^3 + 27B^2)}\beta^2+\frac{(2A^2A' + 9BB')}{(4A^3 + 27B^2)}\beta+\frac{(-4A^2B' + 6ABA')}{(4A^3 + 27B^2)},  \vspace{.1cm} 
\end{equation}
and $\phi(\beta)$ must satisfy the equation \vspace{.15cm} 
\begin{equation}{\label{diff:lattes2}}
\;\;\;\; \phi(\beta)'=\frac{(-3/2AB' + 9/4BA')}{(4A^3 + 27B^2)}\phi(\beta)^2+\frac{(2A^2A' + 9BB')}{(4A^3 + 27B^2)}\phi(\beta)+\frac{(-5/2A^2B' + 15/4ABA')}{(4A^3 + 27B^2)}. 
\end{equation} 
\\
\indent We now repeat this argument for $\phi(\beta)$ and $\phi^2(\beta)$. Namely, since $\phi(\beta)$ and $\phi^2(\beta)$ both satisfy a Riccati equation and $[K(\beta):K]> 32$, the proof of Lemma \ref{lem:Riccati} implies that the polynomial $P_{\phi,\phi(\beta)}\in K[x]$ of degree at most eight must be identically zero; to see this, note that \vspace{.1cm}
\[32< [K(\beta):K]=\big[K(\beta):K(\phi(\beta)\big]\cdot \big[K(\phi(\beta)):K\big]\leq 4\cdot \big[K(\phi(\beta)):K\big],\vspace{.1cm}\]
and hence $\phi(\beta)$ has degree at least $9$ over $K$. Therefore, the Riccati coefficients associated to $\phi(\beta)$ and $\phi^2(\beta)$ must also solve the linear system $(M_\phi, R_\phi)$. On the other hand, we have shown that this system has a unique solution, from which we deduce that $\phi(\beta)$ must satisfy the differential equation: \vspace{.05cm} 
\begin{equation}{\label{diff:lattes3}}
\phi(\beta)'=\frac{(-6AB' + 9BA')}{(4A^3 + 27B^2)}\phi(\beta)^2+\frac{(2A^2A' + 9BB')}{(4A^3 + 27B^2)}\phi(\beta)+\frac{(-4A^2B' + 6ABA')}{(4A^3 + 27B^2)}. \vspace{.05cm} 
\end{equation}
We now equate (\ref{diff:lattes2}) and (\ref{diff:lattes3}) and see that $\phi(\beta)$ satisfies the algebraic equation \vspace{.25cm}    
\begin{equation}{\label{alg:lattes1}} 
0=\frac{(-9/2AB' + 27/4BA')}{(4A^3 + 27B^2)}\phi(\beta)^2+\frac{(-3/2A^2B' + 9/4ABA')}{(4A^3 + 27B^2)}. \vspace{.1cm} 
\end{equation} 
However, the leading term above is non-zero since $2AB'-3BA'\neq0$. Hence, (\ref{alg:lattes1}) contradicts the lower bound $[K(\phi(\beta)):K]>8$. In particular, we deduce that if $\beta$, $\phi(\beta)$ and $\phi^2(\beta)$ all satisfy a Riccati equation, then $[K(\beta):K]\leq32$ as claimed.    
\end{proof}
We now proceed with the proof of Theorem \ref{thm:lattes} much as in Section \ref{sec:Main}. Specifically, we will show that the set $N(\phi,\alpha,\epsilon)$, defined in (\ref{set}), is finite for $\alpha=x(P)$ and all $\epsilon>0$ sufficiently small. To do this, note that Lemma \ref{lem:inv} and Lemma \ref{ram:lattes} applied to $Q=\infty$, together imply that for all $m\geq1$ and $n>m$ there exists $\beta_m\in\overline{K}_v$ such that:
\begin{equation}{\label{ineq:lattes1}}
\;\;\;\;\;\;\;\;\;\;\phi^m(\beta_m)=\infty\;\;\;\;\;\;\;\;\;\;\text{and}\;\;\;\;\;\;\;\; \rho(\phi^n(\alpha),\infty)\geq C(\phi,m)\cdot \rho(\phi^{n-m}(\alpha), \beta_m)^2
\end{equation} 
for some constant $C(\phi,m)=C(\phi^m,\infty)$. Now suppose for a contradiction that $N(\phi,\alpha,\epsilon)$ is infinite and that $0<\epsilon\leq\frac{1}{5}$. In what follows, we fix a large integer 
\[m\geq\max\Big\{\log_4\Big(\frac{48}{\epsilon^2}\Big)+1,5\Big\}.\] 
Now, as in $(\ref{ineq4})$, if $\beta_m\neq\infty$, then for all $n\in N(\phi,\alpha, \epsilon)$ sufficiently large, 
\begin{equation}{\label{ineq:lattes2}}
\frac{1}{|a_n|^{\frac{1}{2}+\epsilon^2}}\geq C(\phi,m)\cdot \big|\phi^{n-m}(\alpha)-\beta_m\big|^2 \cdot \frac{\rho(\beta_m,\infty)^2}{2}.
\end{equation}  
This fact follows from Lemma \ref{lem:comp}, $(\ref{ineq:lattes1})$ and the fact that $N(\phi,\alpha, \epsilon)$ is infinite.  

On the other hand, if $\beta_m\in\overline{\mathbb{F}}_q(t)$, $[K(\beta_m):K]\leq 32$ or $\beta_m=\infty$, then the proof of Theorem \ref{thm:lattes} follows the proof of Theorem \ref{thm:Silv} nearly verbatim; see Cases (1), (2) and (4) of the proof of Theorem \ref{thm:Silv} above. As in these cases, the Liouville bound \cite{Mahler} on the diophantine approximation exponent of $\beta_m$ suffices to prove that $N(\phi,\alpha,\epsilon)$ is finite: the fact that $|\phi^{n-m}(\alpha)-\beta_m|$ appears to a square factor in (\ref{ineq:lattes2}) in the Latt\'{e}s case only changes the relevant bounds by a constant factor. Specifically, we obtain bounds of the form 
\[(1/2+\epsilon^2)\cdot h(\phi^n(\alpha))\leq 64 \cdot h(\phi^{n-m}(\alpha))+\kappa_{16}(\phi,m)\]
(or a stronger bound with $64$ replaced with $2$), in any of these special cases. In particular, standard properties of the (dynamical) canonical height $\hat{h}_\phi$ in \cite[Theorem 3.20]{Silv-Dyn} imply that  
\[n\leq \log_4\bigg(\frac{\kappa_{17}(\phi,\alpha,m,\epsilon)}{3}\bigg)+1.\] 
Hence, $N(\phi,\alpha,\epsilon)$ is finite as claimed.  

Therefore, we may assume that $\beta_m\neq\infty$, $\beta_m\not\in\overline{\mathbb{F}}_q(t)$ and $[K(\beta_m):K]>32$. In particular, Lemma \ref{riccati:lattes} implies that at least one of the algebraic functions $\beta_m$, $\phi(\beta_m)$ or $\phi^2(\beta_m)$ does not satisfy a Riccati equation; we now follow Case (3) of the the proof of Theorem \ref{thm:Silv}. Fix $0\leq i\leq2$ so that $\beta_m^{(i)}:=\phi^i(\beta_m)$ does not satisfy a Riccati equation. Then Lemma \ref{lem:lip}, Lemma \ref{lem:comp}, and the bound in (\ref{ineq:lattes1}) imply that 
\[\frac{1}{|a_n|^{\frac{1}{2}+\epsilon^2}}\geq C(\phi,m)\cdot\Res(\phi^i)^2\cdot \big|\phi^{n-m+i}(\alpha)-\beta_m^{(i)}\big|^2 \cdot \frac{\rho(\beta_m^{(i)},\infty)^2}{2}\vspace{.1cm}\]  
for all $n\in N(\phi,\alpha,\epsilon)$ sufficiently large. In particular, we deduce 
\begin{equation}{\label{ineq:lattes3}} 
(1/2+\epsilon^2)\cdot h(\phi^n(\alpha))\leq (d(\beta_m^{(i)})+3)\cdot h(\phi^{n-m+i}(\alpha))+\kappa_{18}(\phi,m,i)\vspace{.1cm}
\end{equation} 
from the Osgood bound for $\beta_m^{(i)}$ in Theorem \ref{Osgood}; here $d(\beta_m^{(i)})$ is the algebraic degree of $\beta_m^{(i)}$ over $K$. Normally, one cannot improve the trivial bound $d(\beta_m^{(i)})\leq \deg(\phi)^{m-i}$ for generic rational functions $\phi$. However, Latt\'{e}s maps are special in many ways. In particular, it follows from properties of division polynomials that 
\begin{equation}{\label{degree}}
d(\beta_m^{(i)})\leq\frac{1}{2}\deg(\phi)^{m-i},
\end{equation} 
whenever $\phi$ is a Latt\'{e}s map associated to multiplication on an elliptic curve \cite[Exercise 3.7]{Silv-Ell}. Therefore, we see that (\ref{ineq:lattes3}) and basic properties of the canonical height imply that
\[(1/2+\epsilon^2)\cdot \hat{h}_\phi(\alpha)\cdot 4^n\leq 1/2\cdot\hat{h}_\phi(\alpha)\cdot 4^n+ 3\cdot \hat{h}_\phi(\alpha)\cdot 4^{n-m+i} + \kappa_{19}(\phi,m,i,\epsilon).\]
After dividing by $\hat{h}_\phi(\alpha)$, a non-negative quantity since $P\in E(K)$ is non-torsion, we see that 
\[4^n\leq \frac{3}{\,\epsilon^2}\cdot 4^{n-m+i}+\kappa_{20}(\phi,\alpha,m,i,\epsilon)\leq \frac{3}{\,\epsilon^2}\cdot 4^{n-m+2}+\kappa_{20}(\phi,\alpha,m,i,\epsilon)  \vspace{.1cm}\] 
since $0\leq i\leq 2$. On the other hand, $m>\log_4(48/\epsilon^2)+1$ by assumption, so that the bound above implies that  
\[4^n\leq 4^{n-1}+\kappa_{20}(\phi,\alpha,m,i,\epsilon).\vspace{.1cm}\]    
However, such an inequality forces $n\leq \log_4\big(\kappa_{20}(\phi,\alpha,m,i,\epsilon)/3\big)+1$. Hence, $N(\phi,\alpha,\epsilon)$ is finite as claimed.  

To finish the proof, we apply the same argument above to the rational map $\pi(x):=1/\phi(1/x)$ and the basepoint $\alpha:=1/x(P)$. Note first that Lemma \ref{ram:lattes} follows immediately for $\pi$ since ramification indices are invariant under conjugation; see \cite[Exercise 1.5]{Silv-Dyn}. Specifically, $e_{\pi^m}(\beta)\leq2$ for all $\beta\in \overline{K}$ satisfying $\pi^m(\beta)=\infty$, follows from Lemma \ref{ram:lattes} applied to $Q=0$. 

Likewise, as was the case for $\phi$, the linear system $(M_\pi, R_\pi)$ has a unique solution vector $(a,b,c,e,f,g)\in K^6$. In particular, if $\beta\in K^{\sep}$ is such that $\beta$ and $\pi(\beta)$ both satisfy a Riccati equation and $[K(\beta):K]> 32$, then \vspace{.075cm}  
\begin{equation*}{\label{diff:lattes4}}
\beta'=\frac{(4A^2B' - 6AA'B)}{(4A^3 + 27B^2)}\beta^2+\frac{(-2A^2A' - 9BB')}{(4A^3 + 27B^2)}\beta+\frac{(6AB' - 9A'B)}{(4A^3 + 27B^2)},  \vspace{.075cm} 
\end{equation*}
and $\pi(\beta)$ must satisfy the equation \vspace{.075cm} 
\begin{equation*}{\label{diff:lattes5}}
\;\;\;\; \pi(\beta)'=\frac{(5/2A^2B' - 15/4AA'B)}{(4A^3 + 27B^2)}\pi(\beta)^2+\frac{(-2A^2A' - 9BB')}{(4A^3 + 27B^2)}\pi(\beta)+\frac{(3/2AB' - 9/4A'B)}{(4A^3 + 27B^2)}. \vspace{.075cm}
\end{equation*}   
Repeating this argument for $\pi(\beta)$, we see that if $\beta\in K^{\sep}$ is such that $\beta$, $\pi(\beta)$ and $\pi^2(\beta)$ all satisfy a Riccati equation and $[K(\beta):K]>32$, then $\pi(\beta)$ satisfies the algebraic equation, \vspace{.1cm}  
\begin{equation*}{\label{alg:lattes}} 
0=\frac{(3/2A^2B' - 9/4AA'B)}{(4A^3 + 27B^2)}\pi(\beta)^2+\frac{(9/2AB' - 27/4A'B)}{(4A^3 + 27B^2)}. \vspace{.075cm} 
\end{equation*} 
However, the leading term above is non-zero since $2AB'-3BA'\neq0$; hence the equation above contradicts the lower bound $[K(\beta):K]>32$. In particular, if the first three terms of $\Orb_\pi(\beta)$ all satisfy a Riccati equation, then $\beta$ has bounded degree over $K$. From here, the integrality estimates for the $\pi$-orbit of $\alpha=1/x(P)$ follow exactly as the estimates for the $\phi$-orbit of $\alpha=x(P)$.   
\end{proof}
\begin{rmk} Most of the proof of Theorem \ref{thm:lattes} holds for general Latt\'{e}s maps. For instance the key facts, Lemma \ref{ram:lattes} and (\ref{degree}), are true in general. On the other hand, it is unlikely that one can establish a version of Lemma \ref{riccati:lattes} without explicit formulas (more than just recursive definitions \cite[Exercise 3.7]{Silv-Ell}). However, it may be possible to side-step this problem using Galois representations. More specifically, the author wonders whether Serre's open image theorem (established by Igusa \cite{Igusa} in this setting) implies that the $x$-coordinates of large $\ell$-powered torsion points cannot satisfy a Riccati equation; here $\ell$ is a prime number coprime to $p$, the characteristic of $K$. Such a statement, coupled with our techniques, likely implies Theorem \ref{thm:lattes} for all Latt\'{e}s maps. Alternatively, a strong form of Siegel's theorem for non-isotrivial elliptic curves in characteristic $p$ likely follows from \cite[Theorem 2]{Voloch} for all isogenies of degree prime to $p$.           
\end{rmk} 
\vspace{.1cm}
\textbf{Acknowledgements:} It is a pleasure to thank Joseph Silverman, Lucien Szpiro, and Felipe Voloch for the useful discussions related to the work in this paper.  

\end{document}